\documentclass[10pt]{article}

\usepackage{amsmath}
\usepackage{amsthm}
\usepackage{amssymb}
\usepackage[all]{xy}
\usepackage{latexsym}

\newcommand{\ds}{\displaystyle}

\newcommand{\ZZ}{\mathbf{Z}}
\newcommand{\CALA}{\mathcal{A}}
\newcommand{\CALC}{\mathcal{C}}

\newcommand{\abgrp}{\mathcal{AG}}

\newcommand{\cok}{\operatorname{coker}}

\newtheorem{proposition}{Proposition}
\newtheorem{corollary}{Corollary}
\theoremstyle{definition}
\newtheorem{definition}{Definition}

\begin{document}

\title{K-theory of $\lambda$-rings and tensorlike functors.}
\author{F.J.-B.J.~Clauwens}

\maketitle

\begin{abstract}
We refine the invariant on
$K_2(A[C_{p_e}]/I_m,(T-1))$
constructed in a previous paper
to one which is an isomorphism for
all $\lambda$-rings $A$.
\end{abstract}

\section{Introduction}
This paper is part of a long term project aimed at the computation
of $K_1(A[G])$ for $G$ a finite abelian group and $A$ any $\lambda$-ring,
for example a polynomial ring over the integers.
To achieve this aim some $K_2$ groups have to be calculated.
In \cite{fcal} an invariant was constructed on
$K_2(A\otimes B,A\otimes I_B)$ for $B=\ZZ[G]/I_m$.
Here $G$ is a cyclic group of order $q=p^e$ and 
the $I_m$ form a system of $\lambda$-ideals 
which is cofinal with the powers of the augmentation ideal $I_B$.
This invariant map had values in an abelian group
which is an explicit expression in $A$ and $\Omega_A$,
the module of K\"ahler differentials for $A$.
It was proved that this invariant is an isomorphism
as long as $\Omega_A$ has no $p$-torsion.
Unfortunately this excludes the case that $A$ is 
itself a group ring of a $p$-cyclic group.
It is the aim of the present paper to refine this invariant slightly
so that it becomes an isomorphism for all $A$.

The computation of $K_2(A[x]/x^n,(x))$
reveals that it can be expressed as a functor $F$
of the homomorphism $\delta\colon A\to\Omega_A$ of abelian groups,
at least for rings $A$ with the structure of $\lambda$-ring.
The computation is done using a logarithm-like
map $L$, see \cite{flog} and \cite{fexp}.

However for $B$ as above even the best choice of logarithm
is not an isomorphism, 
but has a certain small kernel and cokernel.
In \cite{fcal} it is tried to remedy this situation
by combining with a second logarithmic invariant,
which detects elements which are in the kernel
of the original one.
The result is a map to a fibered product of abelian groups,
which is an isomorphism if $\Omega_A$ has no $p$-torsion.
However the resulting functor $F$ has not the nice
`tensorlike'  properties that the original one had.

The idea of this paper is to construct a new functor
enforcing this nice property,
by taking the fibered product not on the level of abelian groups,
but on the level of the functors producing these groups.
Finally we show that this new invariant
is  an isomorphism for all $A$.

\section{Tensorlike functors}

Write $\abgrp$ for the category of abelian groups and homomorphisms.
Let $S$ be a ring,
and write $\CALC_S$ for the category of left $S$-modules.
A right $S$-module $K$
yields a functor $F_K\colon \CALC_S\to\abgrp$
by the formula $F_K(M)=K\otimes_S M$.

We  call a functor $F$ from left $R$-modules to abelian groups
\begin{itemize}
\item
Weakly additive if it commutes with finite direct sums.
This we always assume.
\item
Additive if it is weakly additive and commutes with direct limits.
\item
Weakly tensorlike if it is weakly additive and commutes with right exact sequences.
\item
Tensorlike if it is additive and commutes  with right exact sequences.
This is the case if it commutes with arbitrary colimits,
for example if it has a right adjoint.
\end{itemize}

\begin{proposition}
\label{cover}
Every weakly additive functor $F$ has a natural tensorlike cover.
If it is already weakly tensorlike it is an isomorphism for 
finitely presented modules.
If it is already tensorlike then it is an isomorphism 
for all modules.
\end{proposition}

\begin{proof}
Let $F\colon\CALC_S\to\abgrp$ be weakly additive.
For $M$ a left $S$-module and $a\in M$ there is an $S$-module
homomorphism $\rho_a$ given by $\rho_a(s)=sa$.
This induces a map $F(\rho_a)\colon F(S)\to F(M)$.
By varying $a$ we get a pairing $F(S)\times M\to F(M)$.
From the weak additivity it follows that this map is additive
also in the second entry.
Therefore we get a map $\mu\colon F(S)\otimes M\to F(M)$.

For $M=R$ this yields a right $S$ module structure on $F(S)$.
So we get a map $\overline{\mu}\colon F(S)\otimes_S M\to F(M)$.
This map is an isomorphism for $M=S$, with inverse $x\mapsto x\otimes 1$.
By weak additivity it is therefore an isomorphism for $M=S^m$.
If  $F$ respects right exact sequences then $\overline{\mu}$
is an isomorphism for all finitely presented modules $M$.
If $F$ is even right exact then $\overline{\mu}$ is an isomorphism for all $M$.
\end{proof}

In our applications we have a category $\CALC$ which is not given as a category 
of left $S$-modules, but as a category of diagrams of  abelian groups,
as discussed in the next section \ref{ar}.
However by the Freyd-Mitchell theorem one can realize  this category
as a category of left $S$-modules  for some $S$.
We make this correspondence explicit in proposition \ref{fm}.

Each logarithmic map is a tensorlike functor of the diagram of abelian
groups consisting of $\delta\colon A\to\Omega_A$.
The combination of the logarithms is a map to a pullback
which is no longer a tensorlike functor.
It has however a tensorlike cover, which is constructed in section \ref{cov}.
We then lift our system of logarithms to this cover
by considering the case where $A$ is the universal $\lambda$-ring.
This is done in section \ref{inv}.

\section{A category of arrows.\label{ar}}

Let $R$ be a ring and $\tau\colon R\to R$ an endomorphism.
\begin{definition}
Given a left $R$-module $N$ the $R$-module $N^\tau$
is defined  to be the one with the same additive structure as $N$, 
but with scalar multiplication given by $r\cdot_\tau n=\tau(r)\cdot n$.
\end{definition}

If $M$ is a left $R$-module then a module homomorphism $D\colon M\to N^\tau$
is the same as an additive homomorphism $D\colon M\to N$ such that
$D(r\cdot m)=\tau(r)\cdot D(m)$,
which we call a $\tau$-crossed homomorphism.
Note that a module homomorphism $\beta\colon N_1\to N_2$
induces a module homomorphism $\beta^\tau\colon N_1^\tau\to\ N_2^\tau$.

\begin{definition}
Given $\tau\colon R\to R$ 
the category $\CALA(R,\tau)$ is the one 
with as objects the homomorphisms $D\colon M\to N^\tau$
and as morphisms the commutative diagrams
\begin{equation*}
\xymatrix{
M_1\ar[r]^\alpha\ar[d]^{D_1}&M_2\ar[d]^{D_2}\\
N_1^\tau\ar[r]^{\beta^\tau}&N_2^\tau\\
}
\end{equation*}
where $\alpha\colon M_1\to M_2$ and $\beta\colon N_1\to N_2$ are $R$-module homomorphisms.
\end{definition}

\begin{definition}
Given a ring $R$ and an endomorphism $\tau$
we construct a new ring $S(R,\tau)$
by starting with the ring $R\langle u,v\rangle$
of noncommutative polynomials 
and dividing out by the relations
$uv=0$, $v^2=0$, $vu=v$, $u^2=u$, $ur=ru$ and $vr=\tau(r)v$.
\end{definition}

Given $R$ modules $M$ and $N$ and a homomorphism $D\colon M\to N^\tau$
we construct a left $S(R,\tau)$-module $P$ by
$P=M \oplus N$,
$u(m\oplus n)=m\oplus 0$,
$v(m \oplus n)=0 \oplus D(m)$,
$r(m \oplus n)=rm \oplus rn$.

\begin{proposition}
\label{fm}
This produces  an equivalence between the category $\CALA(R,\tau)$
and the category of left $S(R,\tau)$ modules.
\end{proposition}

\begin{proof}
If $\alpha\colon M_1\to M_2$ and $\beta\colon N_1\to N_2$
are $R$-homomorphisms such that $\beta\circ D=D\circ\alpha$ then we construct 
a $S(R,\tau)$-homomorphism from $P_1=M_1\oplus N_1$ to $P_2=M_2\oplus N_2$
by $\gamma(m\oplus n)=\alpha(m)\oplus\beta(n)$.

On the other hand if $P$ is a left $S(R,\tau)$ module
then we define $M=uP$,
$N=(1-u)P$,
$D(m)=v\cdot m$.
Then $M$ and $N$ are left $R$ modules
and $D$ is a homomorphism from $M$ to $N^\tau$.
Moreover a $S(R,\tau)$ homomorphism $\gamma\colon P_1\to P_2$
gives rise to $R$ homomorphisms $\alpha\colon uP_1\to uP_2$
and $\beta\colon (1-u)P_1\to(1-u)P_2$ such that $\beta\circ D=D\circ\alpha$.
\end{proof}

\bigbreak

We made the definitions with the following special case in mind.
Let $A$ be a $\lambda$-ring.
\begin{itemize}
\item
The abelian group $M=A$ given by addition,
with a $\ZZ[t]$ module structure fiven by $t\cdot a=\psi^p(a)$.
Here $\psi^p$ is the Adams operation.
\item
The abelian group $N=\Omega_A$ of K\"ahler differentials,
with a $\ZZ[t]$ module structure given by $t\cdot \alpha=\phi^p(\alpha)$,
Here $\phi^p$ is the operation introduced in \cite{flog}.
\item
The differential $\delta\colon A\to \Omega_A$.
Since $\delta\psi^p(a)=p\phi^p\delta(a)$ this is a crossed
homomorphism $D$ with respect to the endomorphism $\sigma$ of $\ZZ[t]$
given by $\sigma(t)=pt$.
\end{itemize}

The above construction
defines  a functor $\Delta$
from $\lambda$-rings to  to $\CALA(\ZZ[t],\sigma)$.
By the above theory we get a from this a functor from $\lambda$-rings to left $S(\ZZ[t],\sigma)$-modules.
\bigbreak

We now recall the main result of \cite{fcal}.
Consider the following diagram:
\begin{equation}
\label{maindia}
\xymatrix{
K_2(A\otimes B,A\otimes I)\ar[rr]^{L_4\circ p_{GM}}\ar[d]^{L_2}&&
K_{2,L}(A\otimes(B/I^2),A\otimes I)\ar[d]^{L_3\circ L_4^{-1}}\\
K_{2,L}(A\otimes B,A\otimes I)\ar[rr]^{p_{GL}}&&
K_{2,L}(A\otimes(B/I^2),A\otimes I)\\
}
\end{equation}
Here $K_2$ is relative algebraic $K$-theory,
and for any commutative ring $W$ and ideal $J$
the group $K_{2,L}(W,J)$ is defined as $\cok(\delta^*\colon J\otimes J\to J\otimes_W\Omega_W)$,
where  $\delta^*(a\otimes b)=a\otimes\delta b+b\otimes\delta a$.
For the definitions of the homomorphisms $L_2$, $L_3$, $L_4$, $p_{GL}$ and $p_{GM}$ we refer to \cite{fcal}.
This diagram has the following properties:
\begin{itemize}
\item
The map $p_{GL}$ is surjective.
\item
If an an element in the bottom left group and one in the upper right group
have the same image then they come both from an element in the $K_2$ group.
In other words the $K_2$ group maps surjectively to the fibre product of the $K_{2,L}$ groups.
\item
In the case that $\Omega_A$ has no $p$-torsion then this element is unique.
In other words the diagram is then cartesian.
\end{itemize}
The $K_{2,L}$ groups in this diagram are tensorlike functors of the diagram $\Delta(A)$
and thus of the left $S(\ZZ[t],\sigma)$ module associated to $A$.
However their fibre product is not a tensorlike functor.
The aim of this paper is to construct its tensorlike cover, 
lift our invariants to this new group,
and show that  the new invariant is an isomorphism
for all $A$, without condition on $\Omega_A$.
Thus the kernel of the old combination of invariants is `explained'
by the fact that the tensorlike cover maps not always injectively to the 
fibre product of $K_{2,L}$ groups.

\bigbreak
We now exhibit the tensorlike structure of the $K_{2,L}$ groups in detail.
A few remarks:
\begin{itemize}
\item
For simplicity we assume that $m\geq e$ and $q>2$.
\item
$\epsilon(n)$ is defined as $m+e-i$ where  $p^i\leq n<p^{i+1}$.
\item
$z\in B$ is defined as $T-1$, where $T$ is the generator of the cyclic group $G$.
\end{itemize}

$K_{2,L}(A\otimes B,A\otimes I)$ is a direct sum
$CF_a\oplus CF_b\oplus CF_2\oplus\dots\oplus CF_{q-1}$,
although $CF_a\oplus CF_b$ is called $CF_{\rm low}$ in proposition 4 of \cite{fcal}.
Here the summands are defined as follows:
\begin{equation*}
CF_a=\frac{\Omega_A}{p^{m+e}\Omega_A+p^e\delta A},\qquad
CF_b=\frac{A}{p^e A},\qquad
CF_n=\frac{\frac{\Omega_A}{p^{\epsilon(n)}\Omega_A}\oplus\frac{A}{p^eA}}{\{\delta a\oplus na\}}
\end{equation*}
The correspondence is as follows:
\begin{itemize}
\item
The class of $\alpha\in\Omega_A$ in $CF_a$ maps to the class of $z\otimes\alpha$
in $K_{2,L}(A\otimes B,A\otimes I)$.
\item
The class of $a\in A$ in $CF_b$ maps to  $z^{q-1}\otimes a\delta z$.
\item
The class of $\alpha\oplus a$ in $CF_n$ maps to
$z^n\otimes\alpha+z^{n-1}\otimes a\delta z$.
\end{itemize}
Similarly $K_{2,L}(A\otimes(B/I^2),A\otimes I)$ is a direct sum
$DF_1\oplus DF_2$.
Here the summands are defined as follows:
\begin{equation*}
DF_1=\frac{\Omega_A}{p^e\Omega_A},\qquad
DF_2=\frac{A}{p^eA+2A}
\end{equation*}
In particular $DF_2$ vanishes if $p>2$ and is $\frac{A}{2A}$ if $p=2$.
The correspondence is as follows:
\begin{itemize}
\item
The class of $\alpha\in\Omega_A$ in $DF_1$ maps to  $z\otimes\alpha$
in $K_{2,L}(A\otimes(B/I^2),A\otimes I)$.
\item
The class of $a\in A$ in $DF_2$ maps to  the class of $z\otimes a\delta z$.
\end{itemize}
The map $p_{GL}$ maps $CF_a$ in the obvious way to $DF_1$
and maps $CF_2$ in the obvious way to $DF_2$;
it vanishes on $CF_b$ and on the $CF_n$ with $n>2$.
The map $L_3(L_4)^{-1}$ maps $DF_1$ to itself by $1-\phi^p$
and maps $DF_2$ to itself by $1-\psi^p$.

\bigbreak

From the above enumeration we see the only nontrivial parts in
the determination of the tensorlike cover are the parts involving
$CF_a$ and $CF_2$, and the latter only if $p=2$.

\section{Determination of the covers.\label{cov}}

Motivated by the above analysis we define the following functors
on the category $\CALA(\ZZ[t],\sigma)$.
\begin{equation*}
CF_a(D\colon M\to N)=\frac{N}{p^{m+e}N+p^eD(M)}\;,\qquad
DF_1(D\colon M\to N)=\frac{N}{p^eN}
\end{equation*}
There is an obvious transformation $\pi\colon CF_a\to DF_1$.
We write $\chi\colon DF_1\to DF_1$ for the transformation given by multiplication by $1-t$.

We also define the following functor with the aim of showing that it is 
the fibre product of $\pi$ and $\chi$:
\begin{equation*}
EF_1(D\colon M\to N)=
\frac{N\oplus \frac{N}{p^mN+D(M)}}{\{p^e \gamma,-(1-t)\gamma\}}=
\frac{\frac{N}{p^{m+e}N}\oplus \frac{N}{p^mN+D(M)}}{\{p^e \gamma,-(1-t)\gamma\}}
\end{equation*}
\begin{proposition}
Let $\pi^*\colon EF_1\to DF_1$ be the projection on the first summand,
and let $\chi^*\colon EF_1\to CF_a$ be given by
$\chi^*(\alpha,\beta)=(1-t)\alpha+p^e\beta$.
Then the diagram of tensorlike functors
\begin{equation*}
\xymatrix{
EF_1\ar[r]^{\pi^*}\ar[d]^{\chi^*}&DF_1\ar[d]^\chi\\
CF_a\ar[r]^\pi&DF_1\\
}
\end{equation*}
is cartesian in the sense that the corresponding digram of right $S(\ZZ[t],\sigma)$-modules is cartesian.
\end{proposition}

\begin{proof}
The functors $CF_a$ and $DF_1$ and $EF_1$ can be viewed as functors
on the category of left $S(\ZZ[t],\sigma)$-modules.
By proposition \ref{cover} we only must show that the diagram of abelian groups 
which one gets by applying those functors to $S=S(\ZZ[t],\sigma)$ itself is cartesian.
The diagram corresponding to $S$ is the direct sum of two diagrams:
one has $M=0$ and $N=\ZZ[t]$ and the other has $M=\ZZ[t]$ and $N=\ZZ[t]$.
We only check the second case: the first is similar but easier.
Thus we have to consider the following diagram:
\begin{equation*}
\xymatrix{
\frac{\ds \ZZ[t]\oplus\frac{\ZZ[t]}{p^m\ZZ[t]+\{g(pt)\}}}{\ds \{p^ef\oplus-(1-t)f\}}\ar[rrr]^{\pi^*}\ar[d]^{\chi^*}&&&
\frac{\ds\ZZ[t]}{\ds p^e\ZZ[t]}\ar[d]^\chi\\
\frac{\ds \ZZ[t]}{\ds p^{m+e}\ZZ[t]+p^e\{h(pt)\}}\ar[rrr]^\pi&&&
\frac{\ds\ZZ[t]}{\ds p^e\ZZ[t]}\\
}
\end{equation*}
To check that this diagram is cartesian is elementary.
\end{proof}

\bigbreak

Again motivated by the analysis in the last section
we define the following functors in the case $p=2$:
\begin{equation*}
CF_2(D\colon M\to N)=\frac{     \frac{N}{2^{m+e-1}N}\oplus\frac{M}{2^eM}    } {\{Da\oplus 2a\}},\qquad
DF_2(D\colon M\to N)=\frac{M}{2M}
\end{equation*}
There is an obvious transformation $\pi\colon CF_2\to DF_2$.
We write $\chi\colon DF_2\to DF_2$ for the transformation given by multiplication by $1-t$.

We also define the following functor with the aim of showing that it is 
the fibre product of $\pi$ and $\chi$:
\begin{equation*}
\begin{split}
EF_2(D\colon M\to N)
&=\frac{\frac{N}{2^{m+e-1}N}\oplus M\oplus\frac{M}{2^{e-1}M}}{\{D(c)\oplus 2h\oplus(c-(1-t)h)\}}\\
&=\frac{\frac{N}{2^{m+e-1}N}\oplus \frac{M}{2^eM}\oplus\frac{M}{2^{e-1}M}}{\{D(c)\oplus 2h\oplus(c-(1-t)h)\}}\\
\end{split}
\end{equation*}

\begin{proposition}
Let $\pi^*\colon EF_2\to DF_2$ be the projection on the second summand,
and let  $\chi^*\colon EF_2\to CF_2$ be given by
$\chi^*(\alpha\oplus a\oplus b)=\alpha\oplus((1-t)a+2b)$.
Then the diagram of tensorlike functors
\begin{equation*}
\xymatrix{
EF_2\ar[r]^{\pi^*}\ar[d]^{\chi^*}&DF_2\ar[d]^\chi\\
CF_2\ar[r]^\pi&DF_2\\
}
\end{equation*}
is cartesian in the sense that the corresponding digram of right $S(\ZZ[t],\sigma)$-modules is cartesian.
\end{proposition}

\begin{proof}
As in the proof of the last proposition we only have to check the case $M=\ZZ[t]$ and $N=\ZZ[t]$.
Thus we have to consider the following diagram:
\begin{equation*}
\xymatrix{
\frac{\ds \frac{\ZZ[t]}{2^{m+e-1}\ZZ[t]}\oplus\ZZ[t]\oplus\frac{\ZZ[t]}{2^{e-1}\ZZ[t]}}{\ds \{c(pt)\oplus2h(t)\oplus(c(t)-(1-t)h(t))\}}
\ar[rrr]^{\pi^*}\ar[d]^{\chi^*}&&&
\frac{\ds\ZZ[t]}{\ds 2\ZZ[t]}\ar[d]^\chi\\
\frac{\ds \frac{\ZZ[t]}{2^{m+e-1}\ZZ[t]}\oplus\frac{\ZZ[t]}{2^e\ZZ[t]}}{\ds \{g(2t)\oplus2g(t)\}}\ar[rrr]^\pi&&&
\frac{\ds\ZZ[t]}{\ds 2\ZZ[t]}\\
}
\end{equation*}
To check that this diagram is cartesian is elementary.
\end{proof}

\begin{corollary}
The kernel of the surjective map from $EF_1(D\colon M\to N)$
to the fibre product of $CF_a(D\colon M\to N)$ and $DF_1(D\colon M\to N)$
is a quotient of the $p^e$-torsion in $N$.\\
The kernel of the surjective map from $EF_2(D\colon M\to N)$
to the fibre product of $CF_2(D\colon M\to N)$ and $DF_2(D\colon M\to N)$
is a quotient of the $2$-torsion in $M$.\\
\end{corollary}

\begin{proof}
Straightforward.
\end{proof}

Write $FP(A)$ for the fibre product of the bottom arrow and right arrow in diagram (\ref{maindia}).
Thus the old invariant can be seen as a map 
$K_2(A\otimes B,A\otimes I)\to FP(A)$.
Write $TC(A)$ for its tensorlike cover viewing $FP(A)$ as a functor 
of the diagram $\Delta(A)$ associated to $A$.
Thus the new invariant to be constructed can be seen as a map 
$K_2(A\otimes B,A\otimes I)\to TC(A)$.

Let us call a ring $A$ torsionfree if there is no $p$-torsion in $A$ nor in $\Omega_A$.
We have seen that the canonical map $TC(A)\to FP(A)$ is an isomorphism if $A$ is torsionfree.

\begin{corollary}
If $p>2$ then $TC(A)$ is isomorphic to 
\begin{equation*}
EF_1(A)\oplus CF_b(A)\oplus CF_2(A)\oplus CF_3(A)\oplus\dots CF_{q-1}(A).
\end{equation*}
If $p=2$ then $TC(A)$ is isomorphic to 
\begin{equation*}
EF_1(A)\oplus CF_b(A)\oplus EF_2(A)\oplus CF_3(A)\oplus\dots CF_{q-1}(A).
\end{equation*}
\end{corollary}

\section{The invariant map and its inverse.\label{inv}}

Recall that there exists a $\lambda$-ring $U$ and an element $u\in U$ with the following property:
for any $\lambda$-ring $A$ and any $x\in A$ there is a unique 
$\lambda$-ring homomorphism $j_x\colon U\to A$ such that $j_x(u)=x$.
Similarly $U_d=U^{\otimes d}$ has the following property
for any $\lambda$-ring $A$ and any $x_1,\dots,x_d\in A$ there is a unique 
$\lambda$-ring homomorphism $j_{x_1,\dots,x_d}\colon U_d\to A$ 
such that $j_{x_1,\dots,x_d}(u_i)=x_i$ for $i=1,\dots,d$.
Here $u_i=1\otimes\dots\otimes u\otimes\dots\otimes1$.
The ring $U$ is a polynomial ring over $\ZZ$ in infinitely many variables;
in particular the $U_d$ are torsionfreee.

\bigbreak

The key to the construction of the new invariant map 
is the following principle:
\begin{proposition}
\label{invmap}
Let $F$ and $G$ be functors from the category of $\lambda$-rings
to the category of abelian groups.
Let $T\colon F\to G$ be a transformation defined on the full
subcategory consisting of the rings $U_d$.
Suppose that $F(A)$ is given by generators
and relations associated to sequences of elements in $A$.
Then $T$ can be extended to all $\lambda$-rings.
\end{proposition}

\begin{proof}
To simplify the presentation of the proof
we consider the special case that $F(A)$
has a generator $\langle a,b\rangle$ for each $a,b\in A$,
and that there is only one type of relation,
say $\langle a,bc\rangle=\langle ab,c\rangle+\langle ac,b\rangle$ for each $a,b,c\in A$.
The general case would involve lots of extra indices, 
which would only obfuscate the idea of the proof.

In order to construct $T_A\colon F(A)\to G(A)$
we define it on generators by
$T_A\langle a,b\rangle =G(j_{a,b})T_{U_2}\langle u_1,u_2\rangle$.
We must then check that it maps the relations to zero.

The description in terms of generators and relations also applies to $F(U_3)$;
in particular $\langle u_1,u_2u_3\rangle=\langle u_1u_2,u_3\rangle+\langle u_1u_3,u_2\rangle$
in $F(U_3)$.
Now we have
\begin{equation*}
\begin{split}
T_A\langle a,bc\rangle
&=G(j_{a,bc}) T_{U_2}\langle u_1,u_2\rangle
=G(j_{a,b,c}\circ j_{u_1,u_2u_3}) T_{U_2}\langle u_1,u_2\rangle\\
&=G(j_{a,b,c}) G( j_{u_1,u_2u_3}) T_{U_2}\langle u_1,u_2\rangle\\
&=G(j_{a,b,c}) T_{U_3} F( j_{u_1,u_2u_3}) \langle u_1,u_2\rangle
=G(j_{a,b,c}) T_{U_3} \langle u_1,u_2u_3\rangle
\end{split}
\end{equation*}
and similarly
\begin{equation*}
\begin{split}
&T_A\langle ab,c\rangle
=G(j_{a,b,c}) T_{U_3} \langle u_1u_2,u_3\rangle\\
&T_A\langle ac,b\rangle
=G(j_{a,b,c}) T_{U_3} \langle u_1u_3,u_2\rangle\\
\end{split}
\end{equation*}
Thus the desired result follows by applying $G(j_{a,b,c})\circ T_{U_3}$
to the relation in $F(U_3)$.

We check that the extended  $T$ is still a transformation.
Let $f\colon A_1\to A_2$ be a $\lambda$-ring homomorphism.
For $a,b\in A_1$ we have
\begin{equation*}
\begin{split}
G(f) T_{A_1}\langle a,b\rangle
&=G(f) G(j_{a,b}) T_{U_2}\langle u_1,u_2\rangle
=G(j_{f(a),f(b)}) \langle u_1,u_2\rangle\\
&=T_{A_2}\langle f(a),f(b)\rangle
=T_{A_2} F(f)\langle a,b\rangle
\end{split}
\end{equation*}
since $f\circ j_{a,b}=j_{f(a),f(b)}$.
Since $G(f) T_{A_1}$ and $T_{A_2}F(f)$ agree on generators 
they agrre on $F_1(A)$.
\end{proof}

\begin{proposition}
\label{invinv}
Let $F$ and $G$ be functors from the category of $\lambda$-rings
to the category of abelian groups.
Suppose that $F$ and $G$ are both given by generators
and relations.
Let $T\colon F\to G$ and $S\colon G\to F$ be transformations
such that $TS=1$ and $ST=1$ on the rings $U_d$.
Then $TS=1$ and $ST=1$ on all $\lambda$-rings.
\end{proposition}

\begin{proof}
Again to simplify the proof we  assume that $F(A)$
has a generator $\langle a,b\rangle$ for each $ab\in A$,
and that $G(A)$ has a generator $\lfloor a,b\rfloor$ for each $a,b\in A$.
We make no assumptions an the form of the relations.
We will prove that $ST=1$ on $F(A)$.

The description in terms of generators and relations also applies to $G(U_2)$;
therefore $T_{U_2}\langle u_1,u_2\rangle$ can be expressed in these
generators, say
\begin{equation*}
T_{U_2}\langle u_1,u_2\rangle
=\sum_{i=1}^r c_i\lfloor v_i,w_i\rfloor,\qquad\text{with }c_i\in\ZZ
\end{equation*}
This assumption implies that
\begin{equation*}
T_A\langle a,b\rangle
=G(j_{a,b}) \sum_{i=1}^r c_i\lfloor v_i,w_i\rfloor
=\sum_{i=1}^r c_i\lfloor x_i,y_i\rfloor
\end{equation*}
where $x_i=j_{a,b}(v_i)$ and $y_i=j_{a,b}(w_i)$.
Therefore
\begin{equation*}
\begin{split}
S_AT_A\langle a,b\rangle
&=S_A\sum_{i=1}^r c_i \lfloor x_i,y_i\rfloor
=\sum_{i=1}^r c_i S_A G(j_{a,b}) \lfloor v_i,w_i\rfloor\\
&=\sum_{i=1}^r c_i  F(j_{a,b}) S_{U_2}\lfloor v_i,w_i\rfloor
=F(j_{a,b}) S_{U_2}\sum_{i=1}^r c_i  \lfloor v_i,w_i\rfloor\\
&=F(j_{a,b}) S_{U_2} T_{U_2}\langle u_1,u_2\rangle
=F(j_{a,b})\langle u_1,u_2\rangle=\langle a,b\rangle
\end{split}
\end{equation*}
Since $S_AT_A$ is the identity on generators,
it is the identity on $F(A)$.
\end{proof}

The aim is now to use proposition \ref{invmap} to construct 
our new invariant map from $K_2(A\otimes B,A\otimes I)$ to $TC(A)$,
and proposition \ref{invinv} to construct its inverse.
The point is of course that we already know 
what it should do on the rings $U_d$ and indeed on torsionfree $\lambda$-rings
since there its should agree with the old invariant map
under the isomorphism $TC(A)\to FP(A)$.
The only thing left to do is to point out why 
$K_2(A\otimes B,A\otimes I)$ and $TC(A)$
are described by generators and relations.

For $K_2$ this well known:
in \cite{ms} its is proved that for a commutative ring $W$ 
and a nilpotent ideal $J$ such that the projection $W\to W/J$ splits
the group $K_2(W,J)$ has a presentation
with as generators $\langle a,b\rangle$
with $a\in W$ and $b\in I$ or $a\in I$ and $b\in J$.
The relations are
\begin{itemize}
\item
$\langle a,b\rangle+\langle b,a\rangle$
for $b\in J$ or $a\in J$ .
\item
$\langle a,b\rangle+\langle a,c\rangle-\langle a,b+c-abc\rangle$
for $a\in J$ or and $b,c\in J$.
\item
$\langle a,bc\rangle-\langle ab,c\rangle-\langle ac,b\rangle$
for $a\in J$ or $b\in J$ or $c\in J$.
\end{itemize}
So for $W=A\otimes B$ and $J=A\otimes I$ we have generators
\begin{equation*}
\langle a_0+a_1z+a_2z^2+\dots+a_{q-1}z^{q-1},b_0+b_1z+b_2z^2+\dots+b_{q-1}z^{q-1}\rangle
\end{equation*}
with $a_i,b_i\in A$ and $a_0=0$ or $b_0=0$.
This can be seen as a generator depending on $2q-1$ elements of $A$.
The relations should of course be rewritten accordingly.
Also one should not forget the relation that two such generators agree
if their $a_i$ or their $b_i$ are equal module the order of $z^i$,
which is $\epsilon(i)$ for $i>0$.

The group $TC(A)$ is a direct sum of summands which each have
a description in terms of generators and relations.
For example $\Omega$ has a presentation with generators
$\lfloor a,b\rfloor$ for $a,b\in A$;
the relations are
\begin{itemize}
\item
$\lfloor a,b\rfloor+\lfloor a,c\rfloor-\lfloor a,b+c\rfloor$.
\item
$\lfloor a,c\rfloor+\lfloor b,c\rfloor-\lfloor a+b,c\rfloor$.
\item
$\lfloor a,bc\rfloor-\lfloor ab,c\rfloor-\lfloor ac,b\rfloor$.
\end{itemize}
The generator $\lfloor a,b\rfloor$ corresponds with the element $a\delta b$ of $\Omega_A$.
The image of $a\delta b$ under $\phi^p$ is 
\begin{equation*}
\psi^p(a)\phi^p(\delta b)=\psi^p(a)(b^{p-1}\delta b-\delta\theta^p(b))
\end{equation*}
Thus in terms of the above presentation $\phi^p$ maps
$\lfloor a,b\rfloor$ to \begin{equation*}
\lfloor \psi^p(a)b^{p-1},b\rfloor-\lfloor\psi^p(a),\theta^p(b)\rfloor
\end{equation*}
Here $\theta^p$ is the $\lambda$-operation introduced in \cite{flog}.
Using this remark it is now easy to write down presentations
of the groups $CF_n$, $CF_b$, $EF_1$ and $EF_2$.

\bigbreak

\end{document}